\documentclass[11pt]{amsart}

\usepackage{amsmath,amssymb,amsthm}
\usepackage[mathscr]{euscript}
\usepackage{mathtools}

\usepackage[margin=1in]{geometry}

\theoremstyle{definition}
\newtheorem{theorem}{Theorem}[section]
\newtheorem{proposition}[theorem]{Proposition}

\newtheorem{example}[theorem]{Example}
\newtheorem{conjecture}[theorem]{Conjecture}
\newtheorem{lemma}[theorem]{Lemma}
\newtheorem{definition}[theorem]{Definition}

\newcommand{\RR}{\mathbb{R}}

\newcommand{\ZZ}{\mathbb{Z}}

\newcommand{\polyp}{\mathcal{P}}
\newcommand{\init}{\mathrm{in}_{\prec}}
\newcommand{\insn}{\mathrm{in}_{\prec_{S_n}}}

\newcommand{\inrevlex}{\mathrm{in}_{\prec_{\mathrm{rlex}}}}

\newcommand{\lp}{\mathcal{L}_{\polyp}}
\newcommand{\den}{\mathrm{den~}}
\newcommand{\grob}{\mathscr{G}}
\newcommand{\ordersn}{\prec_{S_n}}

\begin{document}

\title{Ehrhart Series of Polytopes Related to Symmetric Doubly-Stochastic Matrices}
\author{Robert Davis}
\address{Department of Mathematics\\
         University of Kentucky\\
         Lexington, KY 40506--0027}
\email{davis.robert@uky.edu}
\date{23 March 2015}

\thanks{
The author is partially supported by a 2013-2014 Fulbright U.S. Student Fellowship.
The author thanks Benjamin Braun, Alexander Engstr\"om, Matthew Zeckner, and the referee for their valuable comments and insights.
}

\begin{abstract}
	In Ehrhart theory, the $h^*$-vector of a rational polytope often provide insights into properties of the polytope that may be otherwise obscured. As an example, the Birkhoff polytope, 
	also known as the polytope of real doubly-stochastic matrices, has a unimodal $h^*$-vector, but when even small modifications are made to the polytope, the same property can be very difficult to prove. 
	In this paper, we examine the $h^*$-vectors of a class of polytopes containing real doubly-stochastic symmetric matrices.
\end{abstract}

\maketitle

\section{Introduction} \label{sec:intro}

	For a rational polytope $\polyp \subseteq \RR^n$ of dimension $d$, consider the counting function $\lp(m) = |m\polyp \cap \ZZ^n|$, where $m\polyp$ is the $m$-th dilate of $\polyp$.
	The \emph{Ehrhart series} of $\polyp$ is
	\[
	E_{\polyp}(t) := 1 + \sum_{m\in \ZZ_{\geq 1}} \lp(m)t^m \, .
	\]
	Let $\den\polyp$ denote the least common multiple of the denominators appearing in the coordinates of the vertices of $\polyp$.
	Combining two well-known theorems due to Ehrhart \cite{Ehrhart} and Stanley \cite{StanleyDecompositions}, there exist values $h_0^*,\ldots,h_k^*\in \ZZ_{\geq 0}$ with $h_0^*=1$ such that
	\[
	E_\polyp(t)=\frac{\sum_{j=0}^kh_j^*t^j}{(1-t^{\mathrm{den~}\polyp})^{d+1}} \, .
	\]
	We say the polynomial $h^*_\polyp(t):=\sum_{j=0}^kh_j^*t^j$ is the \emph{$h^*$-polynomial} of $\polyp$ (sometimes referred to as the $\delta$-polynomial of $\polyp$) and the vector of 
	coefficients $h^*(\polyp)$ is the \emph{$h^*$-vector} of $\polyp$.
	That $E_\polyp(t)$ is of this rational form is equivalent to $|m\polyp\cap \ZZ^n|$ being a quasipolynomial function of $m$ of degree at most 
	$d$; the non-negativity of the $h^*$-vector is an even stronger property. If $\den\polyp \neq 1$ then the form of $E_\polyp(t)$ above may not be fully reduced, 
	yet we still refer to the coefficients of this form when discussing $h^*(\polyp)$. Even more tools are available when $\polyp$ is a lattice polytope, that is, when its
	vertices are integral. 
	
	Recent work has focused on determining when $h^*(\polyp)$ is unimodal, that is, when there exists some $j$ for which $h_0^* \leq \cdots \leq h_j^* \geq \cdots \geq h_k^*$.
	The specific sequence in question may not be of particular interest, but unimodal behavior often suggests an underlying structure that may not be immediately apparent.
	Thus, the proofs of various $h^*$-vectors being unimodal are often more enlightening than the sequences themselves. 
	There are a number of approaches possible for proving unimodality, taken from fields such as Lie theory, algebraic statistics, and others \cite{stanleylogconcave}.
	
	In this paper, we consider a variation of the Birkhoff polytope, which is defined as follows. 
	\begin{definition}
		The {\em Birkhoff polytope} is the set of $n\times n$ matrices with real nonnegative entries such that each row and column sum is 1.
	\end{definition}
	We denote this polytope by $B_n$ and note that it is also often referred to as the polytope of real $n\times n$ doubly-stochastic matrices or the polytope of $n\times n$ magic squares.
	The vertex description of $B_n$ is due to the Birkhoff-von Neumann theorem, which finds that $B_n$ is the convex hull of the permutation matrices. 
	The $h^*$-vector of the Birkhoff polytope is difficult to compute in general, and is known only for $n \leq 9$; its volume only for $n \leq 10$ \cite{beckpixton}. 
	As limited as the data is, it has still been shown that $h^*(B_n)$ is symmetric as well as unimodal \cite{athanasiadisbirkhoff, stanley1973, stanley1976}.
	
	On the other hand, little is known about the polytope $\Sigma_n$ obtained by intersecting $B_n$ with the hyperplanes $x_{ij} = x_{ji}$ for all $i, j$, that is, by 
	requiring the corresponding matrices to be symmetric.
	Nothing is new when $n\leq2$, but complications arise once $n \geq 3$ since the vertices of $\Sigma_n$ are no longer always integral. 
	They are contained in the set
	\[
		L_n = \left\{ \frac{1}{2}(P + P^T) | P \in \RR^{n\times n} \mathrm{~is~a~permutation~matrix}\right\},
	\]
	but $L_n$ is not necessarily equal to the vertices of $\Sigma_n$. A description of the vertices and a generating function for the number of them can be found in \cite{StanleyVol2}.
	In \cite{StanleyGreenBook}, Stanley shows that the dimension of $\Sigma_n$ is ${n \choose 2}$ (whereas the dimension of $B_n$ is $(n-1)^2$); 
	he also shows that the $h^*$-vector of $\Sigma_n$ is symmetric and in \cite{StanleyVol1Ed2} computes $E_{\Sigma_n}(t)$ in a reduced form for some small $n$,
	but it is still unknown whether the $h^*$-vector is unimodal in this case. 

	\begin{definition}
		Denote by $S_n$ the polytope containing all real $n\times n$ symmetric matrices with nonnegative entries such that every row and column sum is 2.
		That is, $S_n$ is the dilation of $\Sigma_n$ by two.
	\end{definition}
	Fortunately, some information about $\Sigma_n$ (such as dimension) is retained by $S_n$, a polytope that is combinatorially equivalent but with integral vertices.

	The main purpose of this paper is to examine what happens when trying to prove that $h^*(S_n)$ is unimodal by adapting the techniques used to prove that $h^*(B_n)$ is unimodal.
	Several key ingredients translate nicely to the context of $S_n$, but mysteries remain when examining its toric ideal and certain Gr\"obner bases of it, notions that will be made more precise in 
	Section~\ref{sec:triangulations}. 
	In this direction, we will show the following.
	
	\begin{theorem}\label{thm:mainthm}
		For all $n$, let $I_{S_n}$ denote the toric ideal of $S_n$. The following properties hold:
		\begin{enumerate}
			\item For any term ordering, every element of the reduced Gr\"obner basis $\grob$ of $I_{S_n}$  with respect to this order consists of binomials, one monomial of which is squarefree.
			\item For any term ordering, every variable in $I_{S_n}$ appears in a degree-two binomial in $\grob$.
			\item There exists a class of term orders $\ordersn$ for which the initial term of each degree-two binomial in $\grob$ is squarefree.
			\item For the term orders $\ordersn$, the initial term $\insn(g)$ of each $g \in \grob$ is cubefree, that is, $\insn(g)$ is not divisible by $t_i^3$ for any variable $t_i$ appearing in $g$.
		\end{enumerate}
	\end{theorem}
	 	
\section{Basic Properties, Symmetry, and Integral Closure} \label{sec:basics}
	
	Although relatively little has been established about the Ehrhart theory of $S_n$, it has still been studied and some basic information is known.
	For $\Sigma_n$, the degrees of the constituent polynomials of its Ehrhart quasipolynomial are known.
	
	\begin{theorem}[Theorem 8.1, \cite{Jia}]\label{jia}
		The Ehrhart quasipolynomial of $\Sigma_n$ is of the form $f_n(t) + (-1)^tg_n(t)$, where $\deg f(t) = {n \choose 2}$ and
		\[
		\deg g_n(t) = 
			\left\{
				\begin{array}{ll}
					{n - 1 \choose 2} - 1 & \mathrm{~if~} n \mathrm{~odd}\\
					{n - 2 \choose 2} - 1 & \mathrm{~if~} n \mathrm{~even}
				\end{array}
			\right. .
		\]
	\end{theorem}
	
	Stanley first proved that the above degrees are upper bounds and conjectured equality \cite{stanley1976}, and the conjecture was proven using analytic methods.
	These degrees provide an upper bound on the degree of $h^*_{\Sigma_n}(t)$; we will provide exact degrees later. 
	Since the Ehrhart series of $S_n$, as a formal power series, consists of the even-degree terms of the monomials appearing in $E_{\Sigma_n}(t)$, 
	we get $\mathcal{L}_{S_n}(t) = f_n(2t) + g_n(2t)$.

	The defining inequalities of our polytopes will be helpful in some contexts. For $S_n$, these are
	\begin{eqnarray*}
		x_{ij} &\geq& 0 \mathrm{~for~all~} 1 \leq i \leq j \leq n,\\
		x_{ij} &=& x_{ji} \mathrm{~for~all~} 1 \leq i < j \leq n,\\
		\sum_{i=1}^n x_{ij} &=& 2 \mathrm{~for~each~} j = 1, \ldots, n.
	\end{eqnarray*}
	The first set of inequalities provided indicate that the facet-defining supporting hyperplanes of $S_n$ are $x_{ij} = 0$: if any of these are disregarded, the solution set strictly increases in size. 
	
	\begin{definition}
		A lattice polytope $\polyp \subseteq \RR^n$ is called {\em integrally closed} if, for every $v \in m\polyp \cap \ZZ^n$, there are $m$ points $v_1, \ldots, v_m \in \polyp \cap \ZZ^n$
		such that $v = v_1 + \cdots + v_m$. 
	\end{definition}
	
	This idea is not to be confused with a normal polytope, in which we instead choose $v$ from $m\polyp \cap (my + N)$ for an appropriate choice of $y \in \polyp\cap\ZZ^n$ and
	$N$ is the lattice
	\[
		N = \sum_{z_1,z_2 \in \polyp \cap \ZZ^n} \ZZ(z_1 - z_2) \subseteq \ZZ^n.
	\]
	In particular, every integrally closed polytope is normal, but not every normal polytope is integrally closed.
	There is more discussion of this difference in \cite{gubeladzeconvexnormality}. 
	It is currently an open problem to determine whether integrally closed polytopes have unimodal $h^*$-vectors. This is unknown even in highly restricted cases, such as if the 
	polytope is reflexive, a simplex, or even both. The last case is explored more in \cite{BraunDavis}.
	
	We first would like to prove that $S_n$ is integrally closed.
	To do so, we must interpret the lattice points of $S_n$ as certain adjacency matrices of graphs.
		
	\begin{proposition}\label{intclosed}
		For all $n$, $S_n$ is integrally closed.
	\end{proposition}
	
	\begin{proof}
		The can be seen as a corollary of a theorem of Petersen's 2-factor theorem. 
		For any $m \in \ZZ_{\geq 0}$, each lattice point $X = (x_{ij}) \in mS_n$ can be interpreted as the adjacency matrix of an undirected $2m$-regular multigraph $G_X$ on distinct vertices $v_1,\ldots,v_n$, 
		with loops having degree 1. 
		We first observe that the total number of loops will be even: if there were an odd number of loops, consider the graph with the loops removed.
		The sum of degrees of the vertices in the resulting graph would be odd, which is an impossibility.
		
		Denote by $V_{odd}(G_X)$ the vertices of $G_X$ with an odd number of loops, and write $|V_{odd}(G_X)| = 2mt + s$, where $t,s$ are nonnegative integers and $s < 2m$. 
		Note in particular that $s$ will be even.
		Construct a new graph $G_Y$ with vertex set $V(G_Y) = \{v_1, \ldots,v_n,w_0,w_1,\ldots,w_t\}$ with the same edges as in $G_X$ with the following modifications:
		\begin{enumerate}
			\item For each $v_i \notin V_{odd}(G_X)$, $v_i$ will have $\frac{1}{2}x_{ii}$ loops in $G_Y$.
			\item For each $v_i \in V_{odd}(G_X)$, $v_i$ will have $\frac{1}{2}(x_{ii} - 1)$ loops and an edge between $v_i$ and the lowest-indexed $w_j$ such that $\deg w_j < 2m$.
			\item Vertex $w_t$ will have $\frac{1}{2}(2m - s)$ loops.
		\end{enumerate}
		This new graph will be $2m$-regular, now counting loops as degree 2. 
		Thus, by Petersen's 2-factorization theorem, $G_Y$ can be decomposed into 2-factors.
		Hence the matrix $Y$ corresponding to $G_Y$ will decompose as the sum of $Y_1,\ldots, Y_m$, each summand a lattice point of $mS_{n+t+1}$.
		
		Now we must ``undo'' the changes we made to $G_X$ to obtain the desired sum. 
		Index the rows and columns by $\{v_1, \ldots,v_n,w_0,w_1,\ldots,w_t\}$. 
		Each edge $v_iw_j$ will appear in some $Y_k$ as a 1 in positions $(v_i,w_j)$ and $(w_j,v_i)$. Replace these entries with 0 and add 1 to entry $(v_i,v_i)$. 
		Denote by $X_k$ the submatrix of $Y_k$ consisting of rows and columns indexed by $v_1, \ldots, v_n$ after any appropriate replacements have been made.
		Each replacement preserves the sum of row/column $v_i$, and applying this to each $Y_k$ leaves any entry $(v_i,w_j)$ as 0, so each $X_k$ is a lattice point of $S_n$.
		Thus $X = \sum X_k$, as desired. 
	\end{proof}
	
	A second useful ingredient in proving that $h^*(B_n)$ is unimodal is proving that it has the following property.

	\begin{definition}
		For a lattice polytope $\polyp \subseteq \RR^n$, denote by $k[\polyp]$ the semigroup algebra
		\[
		k[\polyp] := k[x^az^m | a \in m\polyp \cap \ZZ^{n+1}] \subseteq k[x_1^{\pm 1}, \ldots, x_n^{\pm 1}, z].
		\]
		Then $\polyp$ is called {\em Gorenstein} if $k[\polyp]$ is Gorenstein. 
		More specifically, $\polyp$ is {\em Gorenstein of index $r$} if there exists a monomial $x^cz^r$ for which 
		\[
		k[\polyp^{\circ}] \cong (x^cz^r)k[\polyp].
		\]
	\end{definition}
	
	Having the hyperplane description of a polytope can make it easier to determine if it is Gorenstein, as evidenced by the following lemma.
	
	\begin{lemma}[Lemma 2(iii), \cite{BrunsRomer}]\label{facetdistance}
		Suppose $\polyp$ has irredundant supporting hyperplanes $l_1,\ldots, l_s \geq 0$, where the coefficients of each $l_i$ are relatively prime integers. 
		Then $\polyp$ is Gorenstein (of index $r$) if and only if there is some $c \in r\polyp \cap \ZZ^n$ for which $l_i(c) = 1$ for all $i$.
	\end{lemma}

	Generally, proving the unimodality of an $h^*$-vector is a challenging task. There are more techniques available, though, if we have a Gorenstein polytope, that is,
	if the semigroup algebra $k[\polyp]$ is Gorenstein. A closely related class of polytopes is the following.
	
	\begin{definition}
		A lattice polytope $\polyp$ is called {\em reflexive} if $0 \in \polyp^{\circ}$, that is, $0$ is in the interior of $\polyp$, and its {\em (polar) dual}
		\[
			\polyp^{\Delta} := \{y \in \RR^n : x \cdot y \leq 1 \mathrm{~for~all~} x \in \polyp\}
		\]
		is also a lattice polytope.
		A lattice translate of a reflexive polytope is also called reflexive.
	\end{definition}
	
	It was proven by Hibi \cite{HibiDualPolytopes} that reflexive polytope are exactly the Gorenstein polytopes of index 1.
	This connection has been used to reduce questions about integrally closed Gorenstein polytopes to questions about only the integrally closed reflexive polytopes, as in the following statement.
	
	\begin{lemma}[Corollary 7, \cite{BrunsRomer}]\label{reflexive}
		Suppose $\polyp \subseteq \RR^n$ is a full-dimensional integrally closed Gorenstein polytope with supporting hyperplanes $l_1,\ldots,l_s$ as in Lemma~\ref{facetdistance}.
		Consider lattice points $v_0,\ldots,v_k$ of $\polyp$.
		If these points form a $k$-dimensional simplex and $l_i(v_0 + \cdots + v_k) = 1$ for each $i$, then $\polyp$ projects to an integrally closed reflexive polytope $\mathcal{Q}$ of 
		dimension $n - k$ with equal $h^*$-vector. 
	\end{lemma}
	
	\begin{theorem}\label{neven}
		$S_n$ is Gorenstein if and only if $n$ is even. When $n=2k$, $S_n$ is Gorenstein of type $k$, and $h^*(S_n)$ 
		is the $h^*$-vector of a reflexive polytope of dimension $2k^2-2k+1$. Hence, $\deg h^*_{S_n}(t) = 2k^2 - 2k +1$. 
	\end{theorem}
	
	\begin{proof}
		By Lemma~\ref{facetdistance} and knowing the facet description of $S_n$, we can see that the polytope is Gorenstein by choosing integer matrices of $S_n$ whose sum is the all-ones
		matrix. When $n$ is odd, this is impossible: such a matrix has an odd line sum, whereas any sum of matrices in $S_n$ has even line sum.

		Let $n=2k$. For each $i \in \{1, 2, \ldots, k-1\}$, construct a matrix 
		\[
			\begin{pmatrix}
				a_0 & a_{n-1} & a_{n-2} & \cdots & a_2 & a_1 \\
				a_{n-1} & a_0 & a_{n-1} & \cdots & a_3 & a_2 \\
				a_{n-2} & a_{n-1} & a_0 & \cdots & a_4 & a_3 \\
				\vdots & \vdots & \vdots & \ddots & \vdots & \vdots \\
				a_2 & a_3 & a_4 & \cdots & a_0 & a_{n-1} \\
				a_1 & a_2 & a_3 & \cdots & a_{n-1} & a_0 \\
			\end{pmatrix}
		\]
		by setting $a_i = a_{n-i} = 1$ and $a_j = 0$ for all $j \neq i$.
		Construct one additional matrix by setting $a_0 = a_k = 1$ and $a_j = 0$ for all $j \neq 0, k$.
		Each of the $k$ matrices are symmetric and have pairwise disjoint support by construction.  
		These are therefore vertices of a simplex of dimension $k-1$, and Lemma~\ref{reflexive} provides the reflexivity result.
	\end{proof}
	
	Note that this is not the only class of simplices satisfying the conditions of Lemma~\ref{facetdistance} contained in $S_n$ for even $n$; others may be found. It may be interesting to ask how many 
	such distinct simplices in $S_n$ exist.
	
	\begin{example}
	For $n=6$, we construct the special simplex described above. It has three vertices, which are
	\[
		\begin{pmatrix}
			0 & 1 & 0 & 0 & 0 & 1\\
			1 & 0 & 1 & 0 & 0 & 0\\
			0 & 1 & 0 & 1 & 0 & 0\\
			0 & 0 & 1 & 0 & 1 & 0\\
			0 & 0 & 0 & 1 & 0 & 1\\
			1 & 0 & 0 & 0 & 1 & 0\\
		\end{pmatrix},
		\begin{pmatrix}
			0 & 0 & 1 & 0 & 1 & 0\\
			0 & 0 & 0 & 1 & 0 & 1\\
			1 & 0 & 0 & 0 & 1 & 0\\
			0 & 1 & 0 & 0 & 0 & 1\\
			1 & 0 & 1 & 0 & 0 & 0\\
			0 & 1 & 0 & 1 & 0 & 0\\
		\end{pmatrix},
		\begin{pmatrix}
			1 & 0 & 0 & 1 & 0 & 0\\
			0 & 1 & 0 & 0 & 1 & 0\\
			0 & 0 & 1 & 0 & 0 & 1\\
			1 & 0 & 0 & 1 & 0 & 0\\
			0 & 1 & 0 & 0 & 1 & 0\\
			0 & 0 & 1 & 0 & 0 & 1\\
		\end{pmatrix}.
	\]
	\end{example}
	
	\begin{proposition}\label{nodd}
		If $n=2k+1$, then the first scaling of $S_n$ containing interior lattice points is $\left(\frac{n+1}{2}\right)S_n$. Specifically, the number of interior lattice points in this scaling is the number of
		symmetric permutation matrices, i.e. the number of involutions of the set $\{1, 2, \ldots, n\}$. Thus, $\deg h^*_{S_n}(t) = 2k^2$.
	\end{proposition}
	
	\begin{proof}
		For an interior point, each matrix entry must be positive. However, the matrix of all 1s does not work since this results in an odd line sum. Thus there must be
		a 2 in each row and column as well. Thus by subtracting the all-1s matrix, each lattice point corresponds to a symmetric permutation matrix, that is, an involution. 
		The line sum for the interior lattice points will be $n+1$, and we remember that the line sums of matrices in $S_n$ is 2. 
		
		By Theorem 1.5 of \cite{StanleyDecompositions}, 
		$$E_{(S_n)^{\circ}}(t) = (-1)^{{n \choose 2}}E_{S_n}\left(\frac{1}{t}\right).$$
		When expanded as a power series, the lowest-degree term will be $t^{({n \choose 2} + 1) - d}$, where $d = \deg h^*_{S_n}(t)$.
		The degree of $h^*_{S_n}(t)$ follows. 
	\end{proof}
	
	With these, we can deduce the degrees of $h^*_{\Sigma_n}(t)$ for each $n$.
	
	\begin{proposition}
		For all $n$, $h^*(S_n)$ consists of the even-indexed entries of $h^*(\Sigma_n)$.
		Thus, if $n$ is even, then $\deg h^*_{\Sigma_n}(t) = 2(\deg h^*_{S_n}(t))$, and if $n$ is odd, then $\deg h^*_{\Sigma_n}(t) = 2(\deg h^*_{S_n}(t)) + 1$.
	\end{proposition}
	
	\begin{proof}
		As power series, the coefficient of $t^m$ in $E_{S_n}(t)$ is the same as the coefficient of $t^{2m}$ in $E_{\Sigma_n}(t)$. Recalling Theorem~\ref{jia}, this gives
		$$E_{\Sigma_n}(t) = E_{S_n}(t^2) + t\sum_{m\geq 0}f(m)t^{2m}$$
		for some polynomial $f$. 
		So, as rational functions, the first summand of the above will have entirely even-degree terms in the numerator and the same denominator as the rational form of $E_{\Sigma_n}(t)$. 
		Thus, the second summand, when written to have a common denominator as the first summand, will have entirely odd-degree terms in its numerator.
		Therefore, $h^*(S_n)$ consists of the even-indexed entries of $h^*(\Sigma_n)$.
		
		Since $h^*(S_n)$ is symmetric for even $n$ only, and by Proposition~\ref{nodd}, the degrees of $h^*(\Sigma_n)$ follow.
	\end{proof}
	
\section{Toric Ideals and Regular, Unimodular Triangulations} \label{sec:triangulations}

	For a polytope $\polyp \subseteq \RR^n$ let $\polyp \cap \ZZ^n = \{a_1,\ldots, a_s\}$. We define the {\em toric ideal} of $\polyp$ to be the kernel of the map
	$$\pi: T_\polyp = k[t_1,\ldots,t_s] \to k[\polyp],$$
	where $\pi(t_i) = \left(\prod x^{a_i}\right)z$, using the multivariate notation. This ideal we denote $I_\polyp$. Because the lattice points of $S_n$ correspond to matrices, 
	it will sometimes be more convenient to use the indexing
	$$T_{S_n} = k[t_A | A \in S_n \cap \ZZ^{n\times n}] \mathrm{~and~} k[S_n] = k[x^Az^m | A \in mS_n \cap  \ZZ^{n\times n}],$$
	where we now use 
	$$x^Az^m = \prod_{0 \leq i,j \leq n}x_{ij}^{a_{i,j}}z^m$$
	with $A = (a_{i,j})$. Thus $\pi: T_{S_n} \to k[S_n]$ is given by $\pi(t_M) = x^Mz$.
	
	The toric ideal of a polytope has been widely studied, in large part for its connections to triangulations of the polytope. Various 
	properties of the initial ideal of $I_\polyp$ are equivalent to corresponding properties of the triangulation, with perhaps one of the most well-known connections being the following result.
	
	\begin{theorem}[Theorem 8.9, \cite{sturmfels}]
		Given a monomial ordering $\prec$ on $T_\polyp$, the initial ideal $\init (I_\polyp)$ is squarefree if and only if the corresponding regular 
		triangulation of $\polyp$ is unimodular.
	\end{theorem}
	
	In general, $\inrevlex(I_\polyp)$ cannot be guaranteed to be squarefree. This does not rule out the existence of $\inrevlex(I_\polyp)$
	being squarefree for {\em some} ordering of their lattice points, though this may require much more work; the generators of a toric ideal are notoriously difficult to compute in general.
	The following order we place on the lattice points of $S_n$ experimentally appears to provide enough structure to induce regular, unimodular triangulations.
	
	\begin{definition}\label{def:ordersn}
		We place a total order $<_{S_n}$ on the lattice points of $S_n$ by first setting $M <_{S_n} N$ if $M$ contains more 2s in its entries than $N$. 
		This creates a partial order on the lattice points of $S_n$; from this, any linear extension will result in a total order on the lattice points.
		For the remainder of this paper, we will denote any choice of these total orders by $<_{S_n}$.
		This class of orders induces a class of graded reverse lexicographic term orders $\prec_{S_n}$ on the variables of $T_{S_n}$, specifically $t_M \prec_{S_n} t_N$ if and only if $M <_{S_n} N$.
	\end{definition}
	
	We are now ready to prove Theorem~\ref{thm:mainthm}.
	
	\begin{proof}[Proof of Theorem~\ref{thm:mainthm}]
		First, let $\grob$ be the reduced Gr\"obner basis of $I_{S_n}$ with respect to any ordering. It is known to consist of binomials 
		itself. Suppose $\grob$ has a binomial $u - v$ with both terms containing squares, and $\pi(u) = \pi(v) = x^Az^k$. Note in particular that the variables in $u$ and $v$ 
		are distinct. Suppose $t_M$ and $t_N$ are the variables in the separate terms with powers greater than 1. Then $\pi(t_Mt_N)$ is the average of the points corresponding to $\pi(t_M^2)$ and 
		$\pi(t_N^2)$, thus is subtractable from $A$. By the integral closure of $S_n$, there is some third monomial $b$ 
		such that $\pi(t_Mt_Nb) = x^Az^k$. 
		So $u - t_Mt_Nb$ is in $I_{S_n}$; however, we can factor out $t_M$ from this to get $u - t_Mt_Nb = t_M(u_1 - u_2)$. 
		We may similarly factor $t_N$ from $v - t_Mt_Nb$ to get $t_N(v_1 - v_2)$, which must also be in $I_{S_n}$. 
		Therefore $u_1 - u_2$ and $v_1 - v_2$ must be in $I_{S_n}$ themselves, and $u - v$ can be written as
		\[
			u - v =  u - t_Mt_Nb + t_Mt_Nb - v = t_M(u_1 - u_2) - t_N(v_2 - v_1)
		\]
		which contradicts $\grob$ being reduced. Therefore no binomial in $\grob$ can have both terms containing a square.
		
		For the second property, we must show that, for any lattice point $M \in S_n$, we can find a second lattice point $N \in S_n$ such that $M + N$ can be represented in a second,
		distinct sum. Since these are degree 2, the relation must be recorded in $I_{S_n}$, meaning both terms appear individually in $\grob$ (even if not as part of the same binomial). 
		While this can be proven in terms of matrices, it will be easier to work in terms of graph labelings. 
		 		
		As we saw in Proposition~\ref{intclosed}, each lattice point $M \in S_n$ corresponds to a 2-factor $G_M$, a covering of $n$ vertices so that each vertex is incident to two edges. 
		Thus for each 2-factor $G_M$, we want to find a second 2-factor $G_N$ such that $G_M \cup G_N$ can be written as a union of 2-factors, each distinct from both $G_M$ and $G_N$.
		Each covering is a disjoint union of two possible connected components: first,
		a path, possibly of length 0, whose endpoints also have loops; second, a $k$-cycle for some $k \leq n$. This allows us to break the remainder of the proof into three cases.
		
		First suppose $G_M$ contains a path $v_1, v_2, \ldots, v_k$, $k > 1$, with loops at its endpoints. Set $G_N$ to be the graph agreeing with $G_M$ except on these vertices. 
		Here we place a single loop on each of $v_1$ and $v_k$, an edge between these two vertices, and two loops on each of $v_2, \ldots, v_{k-1}$. 
		The union $G_M \cup G_N$ can be decomposed appropriately as a cycle $v_1, v_2, \ldots, v_k, v_1$ and as two loops on each vertex.
		
		Next suppose that $G_M$ contains no such paths but does contain a cycle $v_1, v_2, \ldots, v_k, v_1$ for some $k \geq 2$. 
		Let $G_N$ be the cover with two loops on each $v_i$.
		Then $G_M \cup G_N$ decomposes as the path $v_1, \ldots, v_k$ with a loop on $v_1$ and $v_k$ as one covering and the other covering as the edge $v_1, v_k$ with loops $v_1, v_1$
		and $v_k, v_k$ along with two loops on all other vertices.
		
		If $G_M$ does not fit into either of the previous cases, then its connected components all consist of two loops on each of the $n$ vertices. 
		Form a new graph $G^{\prime}$ by setting it equal to $G_M$, except for two distinct vertices, $v_1$ and $v_2$.  
		Instead, place two edges between $v_1$ and $v_2$. 
		Then $G^{\prime}$ is also a 2-factor, and $G^{\prime} = G_N$ for some lattice point $N \in S_n$.  
		Moreover, the entries of both $M$ and $N$ consist of only zeros or twos, so their average $A = \frac{1}{2}(M+N)$ is a lattice point of $S_n$ distinct from both $M$ and $N$.
		So,  $G_M \cup G_N = G_A \cup G_A$.	
		This covers all cases, so the corresponding $M$ will always appear in a degree-two binomial of $\grob$.
		
		We restrict to the order $\ordersn$ and fix this order for the remainder of the proof. 
		For the third property, consider $t_Mt_N - t_Xt_Y \in \grob$. 
		Since we know one of the monomials must be squarefree, it is enough to check the case when the other monomial is a square square.
		Suppose $M = N$. This can only occur if $M$ is not a vertex; hence, $M$ is the midpoint of $X$ and $Y$. Thus if any entries of $M$ are 2, the corresponding
		entries of $X$ and $Y$ must also be 2. Since $X$ and $Y$ are distinct, though, they have distinct support. This implies that some entry of $M$ is 1, which arises from
		one of the corresponding entries of $X$ and $Y$ being 0 and the other being 2. So, one of $X$ or $Y$ will contain more twos than $M$, giving us 
		$\insn(t_Mt_N - t_Xt_Y) = -t_Xt_Y$. 
		
		Lastly, consider an arbitrary binomial $u - v$ of degree $k$ from $\grob$. If the initial term is the squarefree term, then it is certainly cubefree.
		Otherwise, the binomial is of the form $t_{A_1}^{a_1}\cdots t_{A_r}^{a_r} - t_{B_1}\cdots t_{B_k}$, with $\insn(u-v) = u = t_{A_1}^{a_1}\cdots t_{A_r}^{a_r}$ and each $a_i \geq 1$.
		Since we are using the order $\ordersn$, one of the variables of $v = t_{B_1}\cdots t_{B_k}$ is less than all variables in $u$; without loss of generality, assume this variable is $t_{B_1}$.
		
		Choose a nonzero entry of $B_1$. There will be some variable $t_{M_1}$ such that $M_1 \in \{A_1,\ldots,A_r\}$ and $M_1$ is also nonzero in the same position.
		Now, choose a nonzero entry of $B_1$ such that the position is zero in $A_1$.
		Then we know there is some variable $t_{M_2}$ such that $M_2 \in  \{A_1,\ldots,A_r\} \setminus \{M_1\}$ and $M_2$ is nonzero in this new position.
		Repeating this process gives a monomial $t_{M_1}\cdots t_{M_s}$ such that $M = M_1 + \cdots + M_s$ is nonzero whenever $B_1$ is nonzero.
		If there are any positions that are 2 in $B_1$ and 1 in $M$, then square a variable of $t_{M_1}\cdots t_{M_s}$ whose corresponding matrix is nonzero in that position.
		Repeat on distinct variables if necessary.
		
		The resulting monomial, which we will call $m_1$, is cubefree, and there is some second monomial $m_2$ such that $m_1 - t_{B_1}m_2 \in I_{S_n}$.
		Because $t_{B_1}$ was chosen to be less than all the variables $t_{A_1},\ldots,t_{A_r}$, we know that $\insn(m_1 - t_{B_1}m_2) = m_1$, which divides
		$t_{A_1}^{a_1}\cdots t_{A_r}^{a_r}$
		Since our chosen binomial is in a reduced Gr\"obner basis, the two must be equal. 
		Therefore, every initial term of a binomial in $\grob$ is cubefree.		
	\end{proof}
	
	If the initial terms of $\grob$ with respect to $\ordersn$ can be proven to be squarefree, then the following conjecture holds.
	
	\begin{conjecture}\label{thm:mainconjecture}
		$S_n$ has a regular, unimodular triangulation, hence $h^*(S_n)$ is unimodal when $n$ is even.
	\end{conjecture}
	
	The second statement of the conjecture would follow due to Theorem~1 of \cite{BrunsRomer}.

	The last part of the previous proof adapts the method used in Theorem~14.8 of \cite{sturmfels} to show that $I_{B_n}$ has a squarefree initial ideal for any reverse lexicographic ordering.
	However, we cannot continue to adapt this proof so simply at this point: although one of the matrices $A_j$ coming from $u$ may be nonzero in a position that $B_1$ is also nonzero,
	the entry may be 1 in $A_j$ and 2 in $B_1$, and there is a priori no indication that any other variable corresponds to a matrix with a nonzero entry in the same position.

\section{Future Directions, Questions, and Conjectures}
	
	Experimental data and the results we have shown lead to some natural questions and conjectures.

	\begin{conjecture} Let $\grob$ be the reduced Gr\"obner basis of $I_{S_n}$, and let $g \in \grob$ with $\deg g \geq 3$.
		\begin{enumerate}
			\item The matrix corresponding to the monomials in $g$ does not have a block form. That is, the corresponding graph is connected.
			\item The matrix corresponding to the monomials in $g$ has a decomposition into lattice points of $S_n$ such that one summand consists of only ones and zeros.
		\end{enumerate}
	\end{conjecture}
	If the second part of this conjecture holds, then Conjecture~\ref{thm:mainconjecture} holds as well.
	
	To prove that an initial term of a binomial is squarefree, one strategy would be to prove that both monomials are squarefree.
	We propose a term order on $T_{S_n}$ that is a refinement of $\ordersn$ and appears to hold this behavior.
	
	\begin{conjecture}
		Set $t_M > t_N$ if the matrix $M$ contains more twos than $N$.
		If neither contains a two, then set $t_M > t_N$ if $M$ contains more zeros. 
		Then create a total order through taking a linear extension as in Definition~\ref{def:ordersn}.
		This refinement induces an order such that $\grob$ consists of binomials of degree at most $n-1$, and the binomials of degree greater than 2 are squarefree in both terms.
	\end{conjecture}
	
	Another modification that can be made to $\Sigma_n$ is the following. 
	Denote by $P_n$ the convex hull of the lattice points in $\Sigma_n$. In general, $P_n$ is neither Gorenstein nor integrally closed. 
	However, based on experimental data, we conjecture the following.
	
	\begin{conjecture}
		For all $n$, $h^*(P_n)$ is unimodal.
	\end{conjecture}
	
	Many methods for showing unimodality aim to show that the $h^*$-vector of a polytope
	is the same as the $h$-vector of a simplicial polytope, which has a symmetric $h$-vector.
	However, another approach is necessary for $P_n$, as well as $S_n$ for odd $n$, since neither are Gorenstein.

	Instead of looking at all lattice points of $S_n$, one can form triangulations using only the vertices. These will not be unimodular triangulations, but they might lead to something interesting.
			
	\begin{conjecture}
		For $n \geq 2$, any reverse lexicographic initial ideal of the toric ideal $I_{S_n}$ (using only the vertices of $S_n$) is
		generated by monomials of degree $3(n-2)$, and its minimal generators are $n$-free. That is, the minimal generators are not divisible by $t_i^n$ for any variable $t_i$.  
	\end{conjecture}
	
	The conjecture is experimentally true for $n=3$ by an exhaustive search. Higher dimensions result in exponentially increasing numbers of vertices, vastly increasing 
	the computational difficulty of experimentation.

\bibliographystyle{plain}
\bibliography{davis}

\end{document}